\documentclass[12pt,a4paper,reqno]{amsart}
\usepackage{epsfig}
\usepackage{enumerate}
\usepackage{amssymb}
\usepackage{amsopn}

\newtheorem{MainTheorem}{Theorem}

\newtheorem{lem}{Lemma}[section]

\newtheorem{teo}[lem]{Theorem}
\newtheorem{prop}[lem]{Proposition}
\newtheorem{rema}[lem]{Remark}

\numberwithin{equation}{section}

\newcommand{\Zbb}{\ensuremath{\mathbb{Z}}}

\newcommand{\Rbb}{\ensuremath{\mathbb{R}}}

\DeclareMathOperator{\divv}{div}

\begin{document}

\title[Upper bounds for the number of
limit cycles]{Upper bounds for the number of
limit cycles of some planar polynomial differential systems}
\author {Armengol Gasull}
\address{ Dept. de Matem\`{a}tiques\\
Universitat Aut\`{o}noma de Barcelona\\ Edifici C\\ 08193 Bellaterra,
Barcelona (Spain)} \email{gasull@mat.uab.cat}

\author{Hector Giacomini}
\address{Laboratoire de Math\'{e}matique et Physique Th\'{e}orique\\
C.N.R.S. UMR 6083\\ Facult\'{e} des Sciences et Techniques\\
 Universit\'{e} de Tours\\
 Parc de Grandmont \\
 37200 Tours, France}
\email{Hector.Giacomini@lmpt.univ-tours.fr}

\thanks{}


\date{March, 2008}
\dedicatory{}
\commby{}


\begin{abstract}
We give an effective method for controlling the maximum number of
limit cycles of some planar polynomial systems. It is based on a
suitable choice of a Dulac function and the application of the
well-known Bendixson-Dulac Criterion for multiple connected regions.
The key point is a new approach to control the sign of the functions
involved in the criterion.
 The method is  applied to several  examples.
\end{abstract}

\maketitle

\section{Main result}

One of the few general methods that allows to give upper bounds for
the number of limit cycles of planar  differential systems
\begin{equation*}\Dot x=P(x,y),\,\quad\Dot y= Q(x,y) \end{equation*}
 is the use of
Dulac functions in multiple connected regions, see
\cite{Cherkas,CG1,CG2, GG, Lloyd, Yamato}. Recall that the primer
idea  is that when the function $\divv(P,Q)$ does not vanish on a
simply connected region ${\mathcal U}\subset\mathbb{R}^2,$ then the
above differential system  has no periodic orbit totally contained
in $\mathcal U.$ We state the general Bendixson-Dulac Criterion in
next section, see Theorem~\ref{A}. The main difficulty for practical
uses of this result is that it is needed to find a (Dulac) function
$g$ such that $\divv(gP,gQ)$ does not vanish on a suitable set. This
paper gives a quite general result for polynomial differential
systems, see Theorem~\ref{MT}. Its proof is based on a ``good"
choice of a Dulac function. As we will see in Section~\ref{se3} this
result provides  a constructive way for giving upper and lower
bounds for the number of limit cycles of a large class of planar
polynomial systems.

Given a polynomial $p(s)\in\mathbb{R}[s],$ we will say that the
couple $(k,w(r))\in\mathbb{R}^+\times\mathbb{R}[r]$ is a {\it Dulac
pair} of $p(s)$ if
\[
p_{k,w}(r):=r p(r^2)w'(r)-2k\left(p(r^2)+ r^2p'(r^2)\right)w(r)<0\quad \mbox{for all}\quad r>0.
\]
As we will see in Lemma \ref{lema} and in the proof of Proposition
\ref{B}, the above inequality implies that the function
$|w(r)|^{-1/k}$ is a Dulac function, in any of the connected
components of $\mathbb{R}^2\setminus\{w(r)=0\},$ for the  systems
that in polar coordinates, $r$ and $\theta,$  writes as $\dot
r=rp(r^2), \,\dot\theta =\hat p(r^2).$  Here  $\hat p$ is any
arbitrary polynomial. We remark that there is no need to introduce
this Dulac pair to study the phase portrait of this  system. For
instance their limit cycles are given by the positive zeros of $p$
that are not  zeros of $\hat p.$ Nevertheless, as we will see in
next theorem, Dulac pairs are useful to study more general systems.

\begin{MainTheorem}\label{MT}
Consider the  polynomial differential system
\begin{equation}\label{eq2}\Dot x=P(x,y),\,\quad\Dot y= Q(x,y), \end{equation}
where  $P$ and $Q$ are  real polynomials of degree $n,$ satisfying
that $P(0,0)=Q(0,0)=0.$ In polar coordinates it writes  as
\begin{align}\label{eq3}
  \dot r&=R(r,\theta):=P(r\cos\theta,r\sin\theta)\cos\theta+Q(r\cos\theta,r\sin\theta)\sin\theta,\nonumber\\
  \dot \theta&=\Theta(r,\theta):=\frac{1}{r}\Big(Q(r\cos\theta,r\sin\theta)\cos\theta
  -P(r\cos\theta,r\sin\theta)\sin\theta\Big).\nonumber
\end{align}
Define the polynomial
\begin{equation*}
p(r^2):=\frac1{2\pi r}\int_0^{2\pi} R(r,\theta)\,d\theta
\end{equation*}
and assume that $p(s)$ has the Dulac pair $(k, w(r)),$ where $w$ is
a polynomial of degree $d.$ For these $k$ and $w$ consider the
function
\begin{align}\label{m}
M(r,\theta,k,w):&=R(r,\theta)w'(r)-k\left(\frac{\partial R(r,\theta)}{\partial r}
+\frac{\partial \Theta(r,\theta)}{\partial \theta}+\frac{R(r,\theta)}{r}\right)w(r)\nonumber\\
&=:\sum_{i=1}^{n+d-1} M_i(\theta,k,w)r^i,
\end{align}
and, for any $i\ge1,$ let $m_i(k,w)$ be such that
$\max_{\theta\in[0,2\pi]}M_i(\theta,k,w)\le m_i(k,w) .$

Let $m^+$ be the number of non-negative roots of $w.$ Then, if the
polynomial \[\Phi_{k,w}(r):=\sum_{i=1}^{n+d-1} m_i(k,w)r^i\] is
negative   for all $r\in(0,\infty),$ system~\eqref{eq2} has at most
$m^+$  limit cycles and  all of them are hyperbolic.
\end{MainTheorem}

Next remarks collect some comments on the above result.

\begin{rema}\label{re0} (i) The above theorem gives a
constructive way of testing when a Dulac function of a system of the form
$\dot r=rp(r^2), \,\dot\theta =\hat p(r^2)$
is also suitable for  system~\eqref{eq2}.

(ii) In Proposition~\ref{nueva} we will prove that there are many
polynomials $p$ for which a Dulac pair exists. Moreover, in
Section~\ref{se3}  we will show many systems of the form~\eqref{eq2}
for which the above theorem can be applied.

(iii) Remark \ref{opt} proves that the upper bound given in the
theorem can not be improved.
\end{rema}

\begin{rema}\label{re1}
(i) It is not difficult to see that
  any differential system fulfilling the
 hypotheses of Theorem \ref{MT} has at least $m^+-2-C$ limit cycles, where $C\ge0$
can be  computed as follows:
Let $0\le r_1<\cdots<r_{m^+}$ be the ordered non-negative roots of
$w(r).$ Define the rings $R_i=\{(x,y)\,:\,
r_{i}<\sqrt{x^2+y^2}<r_{i+1}\},$ $i=1,2,\ldots, m^+-1.$ Then
 $C$  is the number of rings among $R_1,R_2,\ldots,R_{m^+-1}$ containing critical points of system
\eqref{eq2}.

(ii) From the proof of the theorem it follows that there is at most
one limit cycle in each of the sets
$R_1,R_2,\ldots,R_{m^+-1},R_{m^+},$  where $R_{m^+}$ is the
unbounded ring $R_{m^+}=\{(x,y)\,:\, r_{m^+}<\sqrt{x^2+y^2}\}$
 and that (when $r_1>0$) there is no
limit cycle in the disc  $\{(x,y)\,:\, \sqrt{x^2+y^2}<r_{1}\}.$
Moreover in the rings $R_2,R_3,\ldots,R_{m^+-1}$ not containing
critical points of the system there exists always a limit cycle and
its stability is given by the sign of $w$ on them.
 The limit cycle on $R_1$ only can exist when $r_1>0$ or when
 $r_1=0$ and  the sign of $w$ on this set
does not coincide with the stability of the origin.
\end{rema}

Let us see how the theorem works in a concrete example. We will prove that the system
\begin{align}\label{mot}
  \dot x&=-y+4x-\frac{49}{10}x^3-\frac{26}{5}xy^2+\frac15x^2y^2+x^5+2x^3y^2+xy^4=P(x,y),\nonumber\\
  \dot y&=\phantom{-}x+4y-\frac{23}{5}x^2y-5y^3-\frac15xy^3-\frac2{15}y^4+x^4y+2x^2y^3+y^5=Q(x,y)
\end{align}
has exactly two limit cycles. Examples with parameters will be
studied in Section~\ref{se3}. Following the notation of the theorem
we get that $p(s)=4-79s/16+s^2.$ Moreover, by using Proposition
\ref{nueva} we take $w(r)=r^2p'(r^2)=r^2(2r^2-79/16).$ Finally we
choose   $k=4/5.$ We have
\[
p_{k,w}(r)=-\frac{79}{10}r^2 -\frac{1287}{128}r^4
+\frac{237}{40}r^6-\frac{8}{5}r^8.\]   By using Sturm's rule it is
easy to prove that $p_{k,w}(r)<0$ for all $r\ne0$ and so $(k,w(r))$
is a Dulac pair for $p.$ Straightforward computations give that
\begin{align*}
&M(r,\theta,k,w)= -\frac{79}{10}r^2+\frac1{2560}\left(-25740+\frac{16432}{5}\cos(2\theta)
+316\cos(4\theta)\right)r^4\\
&-\frac{79}{4800}\left(12\cos(\theta)+3\cos(3\theta)-15\cos(5\theta)+46\sin(\theta)-
7\sin(3\theta)-5\sin(5\theta)\right)r^5\\
&+\frac1{80}\left(474-\frac{128}{5}\cos(2\theta)-8\cos(4\theta)\right)r^6\\
&+\frac{1}{75}\left(6\cos(\theta)+9\cos(3\theta)-15\cos(5\theta)-2\sin(\theta)+
9\sin(3\theta)-5\sin(5\theta)\right)r^7-\frac85r^8.\\
\end{align*}
By using rough bounds like $$-46\le
6\cos(\theta)+9\cos(3\theta)-15\cos(5\theta)
-2\sin(\theta)+9\sin(3\theta)-5\sin(5\theta)\le 46$$ we get that
\[
M(r,\theta,k,w)\le \Phi_{k,w}(r)= -\frac{79}{10}r^2 -\frac{3459}{400}r^4 +\frac{869}{600}r^5
+\frac{1269}{200}r^6 +\frac{46}{75}r^7 -\frac{8}{5}r^8
\]
and it can be proved by using Sturm's rule that it is negative for
all $r>0.$ Therefore, since $w$ has two non-negative roots, $m^+=2$
and the system has at most two (hyperbolic) limit cycles.

\vspace{0.3cm}

It is clear that the origin is an unstable focus. By taking the
resultant of $P$ and $Q,$ for instance with respect to $y,$  and
applying Sturm's rule to this new polynomial we prove that the
origin is the unique critical point of the system. By studying the
flow of the system on the circles $\{x^2+y^2=R^2\}$ for $R$ big
enough, and on $\{x^2+y^2=r_0^2\},$ where $r_0$ is the positive root
of $w,$ and using the above information we deduce that the two limit
cycles actually exist. Indeed one of them is in
$\mathcal{D}=\{x^2+y^2<r_0^2\}$ and is stable and the other
surrounds $\mathcal D$ and is  unstable.

In fact, we have presented the study of system \eqref{mot} in this
introduction because it  shows at the same time the power and the
limitations of the method. By one hand, for a concrete system which
has no  a priori structure, we can give the exact number of limit
cycles. On the other hand, in all the cases that we have studied,
for polynomials systems of degree $2n$ or $2n+1$ it only has been
applicable for  systems having at most $n$ limit cycles and when all
these limit cycles are hyperbolic and nested.

Due to the extreme difficulty of the second part of the Hilbert's sixteenth problem
it is never easy to give explicit and realistic upper bounds for the number of limit
cycles of a planar polynomial system.
Moreover  most results in the literature  can
only be applied to control the limit cycles of  particular types of
differential systems, namely Li\'{e}nard systems, quadratic systems,
cubic systems, systems with homogeneous nonlinearities, etc, giving
 usually  upper bounds of 0, 1 or 2 limit cycles,
see for instance \cite{Cherkas0,Ye,Z}. The
criterion proved in Theorem~\ref{MT}  is not subject to these
restrictions.

The paper is organized as follows: Section 2 contains some
preliminary results and the proof of Theorem~\ref{MT}. In
Section~\ref{se3}, we will apply this result to different families
of differential systems providing in all of them  explicit upper
bounds for the number of limit cycles. In fact, we present two
examples with limit cycles surrounding a unique critical point (see
Examples~1 and 2), two situations with a limit cycle surrounding
several critical points (see Examples~3 and 4) and finally a family
of examples of differential systems showing that, in the set of
systems of the form~\eqref{eq2}, there are many open subsets of
systems satisfying all the hypotheses of the theorem.

\section{Preliminary results and proof of Theorem A}\label{se2}

Let  $\mathcal U\subset\mathbb{R}^2$ be an open set with smooth
boundary and such that its fundamental group, $\pi_1({\mathcal U}),$
is $\Zbb* \stackrel{\scriptsize\ell}{\cdots}*\Zbb$ or in other words
having $\ell$ gaps.  For short we will say that $\mathcal U$ is an
$\ell${\it-punctured}\, open set and we will denote by
$\ell({\mathcal U})$ its number of gaps. Notice that with this
notation, simply connected sets are 0-punctured sets and multiple
connected sets are $\ell$-punctured sets with $\ell\ge 1.$

The following result is a well known extension of the
Bendixson-Dulac Criterion to $\ell$-punctured sets. For a proof, see
any of the papers \cite{GGi,Lloyd, Yamato}.

\vspace{0.2cm}

\begin{teo}[Extended Bendixson-Dulac Criterion]\label{A} Consider the $\mathcal{C}^1$-differential system
\begin{equation}\label{equat}\Dot x=P(x,y),\,\quad\Dot
y= Q(x,y) \end{equation} and set $X=(P,Q).$ Let $\mathcal U$ be an
open $\ell$-punctured subset  of $ \mathbb{R}^2$ with smooth
boundary. Let $g\colon {\mathcal U}\rightarrow \mathbb{R}$  be a
$\mathcal{C}^1$-function  such that $$
    M:=\divv(gX)=
\frac{\partial g}{\partial x}P+\frac{\partial g}{\partial y}Q +(
\frac{\partial P}{\partial x}+ \frac{\partial Q}{\partial y} )g=
\langle \nabla g,X \rangle +g\divv(X) $$ does not change sign in $U$
and vanishes only on a null measure Lebesgue set  and such that
$\{M=0\} \cap\{g=0\}$ does not contain periodic orbits of
(\ref{equat}). Then the maximum number of periodic orbits of
(\ref{equat}) contained in $\mathcal U$ is $\ell.$ Furthermore each
one  is a hyperbolic limit cycle that does not cut
 $ \{g=0\}$ and its
stability  is given by the sign of $gM$ over it.
\end{teo}

We will prove and use the following corollary of the above theorem.

\begin{prop}\label{B} Consider the $\mathcal{C}^1$-differential system
\begin{equation}\label{bis} \Dot x=P(x,y),\,\quad\Dot
y= Q(x,y) \end{equation} and set $X=(P,Q).$ Assume that there exist
a positive real number $k$ and a polynomial $f(x,y)$ such that
\begin{equation}\label{eme}
    M_k:=  \frac{\partial f}{\partial x}P+\frac{\partial f}{\partial y}Q
  -k( \frac{\partial P}{\partial x}+ \frac{\partial Q}{\partial y} )f=
  \langle \nabla f,X \rangle -kf\divv(X)
\end{equation} does not vanish in an open region with regular
boundary ${\mathcal W}\subset \Rbb^2.$  Let ${\mathcal U}_1,$
${\mathcal U}_2$,$\ldots,$ ${\mathcal U}_m,$  be the connected
components of ${\mathcal W}\setminus\{f=0\}.$ Then
\begin{enumerate}[(i)]

\item The periodic orbits of system \eqref{bis} contained in $\mathcal W$
never cut the set $\{f=0\}.$

\item The maximum number of limit cycles of system~\eqref{bis}
contained in each ${\mathcal U}_j,$ $j=1,2,\ldots, m,$  is
$\ell({\mathcal U}_j),$  and all of them are hyperbolic.
Moreover their stability is given by the sign of $-fM_k$ on each
region.

\item The maximum number of limit cycles of system \eqref{bis} in $\mathcal
W$ is $\ell({\mathcal U}_1)+\ell({\mathcal
U}_2)+\cdots+\ell({\mathcal U}_m)$ and all them are hyperbolic.
\end{enumerate}
\end{prop}

\begin{proof} Since  $M_k$ does not vanish on $\mathcal W$ we know
that
 $ \left.\langle \nabla f,X \rangle\right|_{\{f=0\}\cap{\mathcal W}} $ does not
vanish. Therefore  the periodic orbits of \eqref{bis} which are
totally contained in $\mathcal W$ never  cut the set $\{f=0\}$ and
(i) follows. To prove items (ii) and (iii) it suffices to apply
Theorem \ref{A} to each one of the connected components ${\mathcal
W}\setminus \{f=0\}.$ Note that, on each of them, the function
$g=|f|^{-1/k}$ is smooth and moreover
$$
     \divv(gX)=
    \frac{-\mbox{sign}(f)}{k}|f|^{-1/k-1}\left[\langle \nabla f,X \rangle -k f
     \divv (X)\right]= \frac{-\mbox{sign}(f)}{k}|f|^{-1/k-1} M_k.
$$
Therefore the upper bound for the number of limit cycles follows.

Finally we prove the hyperbolicity of all the   limit cycles. Let
$\gamma=\{(x(t),y(t))\,:\, t\in[0,T]\}$ be one of them. We must
prove that $\int_0^T\divv X(x(t),y(t))\,dt\ne0,$ see for instance
\cite[Thm. 1.23]{DLA}. Since $\gamma$ does not intersect the set
$\{f=0\}$ we have that over $\gamma$
$$ \divv X=-\frac{M_k}{k f}+\frac{\langle \nabla f,X \rangle}{k f}.
$$
Hence
\[
\int_0^T\divv X(x(t),y(t))\,dt=-\int_0^T \frac{M_k(x(t),y(t))}{k f(x(t),y(t))} \,dt\ne0
\]
and its stability is given by the sign of $-fM_k$ on the region
where $\gamma$ lies. Therefore the result follows.
\end{proof}

To apply the above result, it will be useful to get the expression
of the function $M_k$ given in the Proposition~\ref{B} in terms of
the components of the differential system~\eqref{bis}, written in
polar coordinates.

\begin{lem}\label{le2} Let $\dot r=R(r,\theta), \, \dot
\theta=\Theta(r,\theta)$ be the expression of  system~\eqref{bis} in
polar coordinates. Then the function \eqref{eme} writes as

\begin{align}
  M_k&=  \frac{\partial f}{\partial x}P+\frac{\partial f}{\partial y}Q
  -k( \frac{\partial P}{\partial x}+ \frac{\partial Q}{\partial y} )f\nonumber\\
&=\frac{\partial f}{\partial r}R+\frac{\partial f}{\partial \theta}\Theta
-k\left(\frac{\partial R}{\partial r}
+\frac{\partial \Theta}{\partial \theta}+\frac{R}{r}\right)f.\nonumber
\end{align}
\end{lem}

\begin{lem}\label{le1} Let $p(s)$ be a real polynomial having all its roots
real and simple.

(i) For all $s\in \mathbb{R}$
\[
p(s)p''(s)-(p'(s))^2<0.
\]

(ii) For each $k\in \mathbb{R}$ define the new polynomial
\[
q_k(r):=2 r^4\left(p(r^2)p''(r^2)-k(p'(r^2))^2\right)+ 2(1-k)r^2p(r^2)p'(r^2).
\]
Then there exists $k,$ with $|k-1|$ small enough,  such that
$q_{k}(r)<0$ for all $r\ne0.$

\end{lem}

\begin{proof} It is not restrictive to write $p(s)=\prod_{j=1}^l
(s-s_j)$ with $s_i\ne s_j$ for $i\ne j.$ From the equalities

\begin{eqnarray*}
p(s)p''(s)-(p'(s))^2&=&p^2(s)\left(\frac{p'(s)}{p(s)} \right)'=p^2(s)\left(
\frac{\sum_{i=1}^l\left(\prod_{j=1,\,j\ne i}^l(s-s_j)  \right)}
{\prod_{j=1}^l (s-s_j)  }
\right)'\\
&=&p^2(s)\left(
\sum_{i=1}^l \frac 1{s-s_i}
\right)'=-p^2(s)\left(
\sum_{i=1}^l \frac 1{(s-s_i)^2}
\right)<0
\end{eqnarray*}
item (i) follows.

To prove (ii) notice that by item (i)
\[q_1(r)=2
r^4\left(p(r^2)p''(r^2)-(p'(r^2))^2\right)<0\quad\mbox{for all}\quad
r\ne0.\] Moreover $q_k(r)=q_1(r)+(1-k)\widetilde{q}(r)$ for some polynomial $\widetilde{q}$ of degree $4l$ and
\begin{eqnarray*}
q_k(r)&=&2c_0c_1(1-k)r^2+2
\left(c_1^2+4c_0c_2-2k(c_1^2+c_0c_2)\right)r^4\\
&&+\cdots+2\left((1-k)l^2-kl\right)r^{4l},
\end{eqnarray*} where
$p(s)=c_0+c_1s+c_2s^2+\cdots+s^l.$ Note that
$p(0)p''(0)-(p'(0))^2=2c_0c_2-c_1^2<0$ and thus,  for $|k-1|$ small
enough, $(1-k)l^2-kl$ and $c_1^2+4c_0c_2-2k(c_1^2+c_0c_2)$ are both
negative. Hence when $c_0c_1=0,$ taking $|k-1|$ small enough,
 we can always ensure that the
sign of $q_k(r)$ is negative in $\mathbb{R}\setminus\{0\}.$ When $c_0c_1\ne0,$ in addition we have  to take
$k$  such that $c_0c_1(1-k)<0.$
\end{proof}

\begin{prop}\label{nueva}
Let $p(s)$ be a real polynomial. Then
\begin{enumerate}[(i)]
\item If $p(s)$ has all its roots real and simple by taking $w(r)=r^2p'(r^2)$ and some $k$ near 1
then $(k,w(r))$ is a Dulac pair of $p(s).$
\item If $p(s)$ has some real multiple positive root  then it has no Dulac pair.
\item  There are many polynomials $p(s)$ with complex roots also having  Dulac pairs.
\end{enumerate}

\end{prop}
\begin{proof}
(i) Note that
\[
p_{k,w}(r)=rp(r^2)\left(r^2p'(r^2)\right)'-2k\left(p(r^2)+r^2p'(r^2)\right)r^2p'(r^2)=q_{k}(r),
\]
where $q_k$ is given in Lemma \ref{le1}. By using this lemma the result follows.

(ii) Let $s^*>0$ be one of these roots. It is easy to see that  for
any couple $k$ and $w,$ the polynomial $p_{k,w}(r)$ also vanishes at
$\sqrt{s^*}$. Hence the result follows.

(iii)  We construct a class of polynomials for which the result
holds. Let $p(s)$ be a polynomial such that  $p(s)=p_1(s)p_2(s)$
where $p_1(r^2)+r^2p_1'(r^2)>0$ for all $r\in\mathbb{R}$ and
$p_2(s)=\prod_{i=1}^{j}(s-r_i^2),$ being  $r_i^2$  different
positive numbers. Consider $w(r)=p_2(r^2)$ and $k=1.$ Then
\[
p_{1,w}(r)=-2p^2_2(r^2)\left(p_1(r^2)+r^2p_1'(r^2) \right)\le0\quad\mbox{ for all } r\in\mathbb{R}.
\]
Take now $w_\varepsilon(r)=p_2(r^2)+\varepsilon^2p_2'(r^2).$ Then
\[
p_{1,w_\varepsilon}(r)=p_{1,w}(r)+\varepsilon^2W(r)
\]
for some polynomial $W$ of the same degree  that $p_{1,w}.$
Moreover
\[
W(r_i)=-2r_i^2p_1(r_i^2)(p'_2(r_i^2))^2<0
\]
for all $i=1,\ldots,j.$ Hence for $|\varepsilon|$ small enough the
polynomial  $p_{1,w_\varepsilon}(r)$ is negative for all
$r\in\mathbb{R},$ as we wanted to prove.
\end{proof}

\begin{rema}
Notice that, given a polynomial $p$ under suitable hypotheses, item
(i) of Proposition \ref{nueva} provides a constructive way of
finding Dulac pairs. On the other hand
 item (iii)   provides  a theoretical way to see that the same situation holds for other
 polynomials. Nevertheless,  for a given polynomial
$p,$ even with complex roots, it is not difficult to find a Dulac
pair. We give a couple of examples and some intuition of how we get
them.
 \begin{enumerate}[(a)]
\item For $p(s)=(s-2)(s-4)(s^2+4)(s+3)$ take $w(r)=(r^2-1)(r^2-3)$ and $k$ any of the values
$1/2,1$ or $2.$
\item For $p(s)=-35-36 s+49s^2/2-14 s^3/3+ s^4/4,$ which has a unique positive root $s_0\simeq
11.12$ take $k=1$ and $w(r)=r^2-\alpha$ for any $\alpha$ for
instance in the interval $[5.5,11].$
\end{enumerate}

By Remark \ref{re0} the existence of $k>0$ and $w$ such that
$p_{k,w}<0$ for all $r>0$ implies that we can apply Proposition
\ref{B} with $f=w(r).$ Therefore each one of the hyperbolic limit
cycles of system $\dot r=rp(r^2),\,\dot \theta=\hat p(r^2),$ which
are given by the positive simple zeros of $p$, has to be contained
in one of the 1-punctured regions of $\mathbb{R}^2\setminus\{w=0\}.$
For instance in case (b), there is only one limit cycle $r^2\simeq
s_0$ and we can try with a function $w(r)=r^2-\alpha$ with $\alpha$
smaller than this value. We construct the function $w$ of item (a)
in a similar way.
\end{rema}

Next lemma will be useful to get systems for which the hypotheses of
Theorem~\ref{MT} hold and to prove that the upper bound given by the
theorem is optimal.

\begin{lem}\label{lema} Consider  system
\begin{eqnarray}\label{tri}
  \dot x =&x u(x^2+y^2)-y v(x^2+y^2),\nonumber\\
  \dot y =&x v(x^2+y^2)+y u(x^2+y^2),
\end{eqnarray}
where $u$ and $v$ are arbitrary real polynomials. If $u$ has all
their roots real and simple then, taking $w(r)=r^2u'(r^2),$ the function $\Phi_{1,w}(r)$ introduced
in Theorem~\ref{MT} is negative for all $r\in(0,\infty).$
\end{lem}

\begin{proof} Writing the system in polar coordinates we obtain
\begin{eqnarray*}
  \dot r =& R(r,\theta)=& r u(r^2),\\
  \dot \theta =&\Theta(r,\theta)=&v(r^2).
\end{eqnarray*}
Clearly, from the above expression we get $p=u,$ where $p$ is the
polynomial introduced in Theorem~\ref{MT}. Thus, taking $k=1$ and $w(r)=r^2u'(r^2),$ we obtain
\begin{align*}
M(r,\theta,1,w)&=R(r,\theta)\left(r^2u'(r^2) \right)'-\left(\frac{\partial R(r,\theta)}{\partial r}
+\frac{\partial \Theta(r,\theta)}{\partial \theta}+\frac{R(r,\theta)}{r}\right)r^2u'(r^2)\nonumber\\
&=r u(r^2)\left(r^2u'(r^2)\right)' -\left(\left(ru(r^2)\right)'
+u(r^2)\right)r^2u'(r^2)\nonumber\\
&=2r^4\left(u(r^2)u''(r^2)-\left(u'(r^2)\right)^2\right)=\Phi_{1,w}(r).
\end{align*}
By using Lemma~\ref{le1}.(i) the result follows.
\end{proof}

\begin{rema}\label{opt} Notice that the above lemma implies that the upper
bound  given in Theorem~\ref{MT} can not be improved. Consider in
system~\eqref{tri} a polynomial $u$  such that  all
their roots are real and has no common roots with $v.$ Then system~\eqref{tri} has
as many limit cycles (indeed invariant circles)
 as number  of positive roots, say $m^*$. It is easy to take a polynomial $u$
 under the above hypotheses  such that
 $u'$ has exactly $m^*-1$ positive roots. Hence the number of non-negative roots
 of $w(r)=r^2p'(r^2)=r^2u'(r^2)$ is $m^+=m^*.$ By applying
 Theorem~\ref{MT} we get an upper bound of $m^+=m^*$ limit cycles,
 which  is indeed the actual number of limit cycles of the system.
\end{rema}

\begin{proof}[Proof of Theorem \ref{MT}]
We apply Proposition~\ref{B} to system \eqref{eq2}  with
$f(r,\theta)=w(r)$ and the value $k$ given in the statement of
the Theorem.  Then, by using Lemma~\ref{le2}, we get that the
expression of $M_k$ given in Proposition~\ref{B}  is
\begin{equation*}
  M_k= R(r,\theta)w'(r)
-k\left(\frac{\partial R(r,\theta)}{\partial r}
+\frac{\partial \Theta(r,\theta)}{\partial \theta}+\frac{R(r,\theta)}{r}\right)w(r)=M(r,\theta,k,w),
\end{equation*}
where $M(r,k,\theta,w)$ is the function given in  Theorem~\ref{MT}.

Moreover, by hypothesis, we have
\[
 M_k=M(r,\theta,k,w)=\sum_{i=1}^{n+d-1} M_i(\theta,k,w)r^i\le
  \sum_{i=1}^{n+d-1} m_i(k,w)r^i=\Phi_{k,w}(r)<0
\]
for all $r\in(0,\infty).$ Therefore, by Proposition~\ref{B}, the
maximum number of limit cycles can be bounded above by studying the
topology of the connected components of the set ${\mathcal
W}:=\mathbb{R}^2\setminus\{w(r)=0\}.$ Clearly it has $m^+$ connected
components, all of them  indeed 1-punctured when $w(0)=0$ and it has
$m^++1$ connected components when $w(0)\ne0$. Notice that in this
later case one of them is simply connected and the other $m^+$ are
1-punctured sets. In any case,  again by Proposition~\ref{B}, there
is no limit cycle in the simply connected component and  there is at
most one limit cycle in each of the 1-punctured components, which is
hyperbolic when it exists.  Moreover, since $M_k<0,$ its stability
is given by the sign of $w$ on each component.  As it is already
said in Remark~\ref{re1}.(ii), it can be proved that the bounded
1-punctured components of $\mathcal W$ not having critical points of
system~\eqref{eq2} contain effectively a limit cycle. This result
holds because each one of these  rings is either positively or
negatively invariant by the flow of the system.
\end{proof}

\section{Examples}\label{se3}

\subsection{Example 1}
Consider the system
\begin{align}\label{eqex1}
  \dot x&=x(1-(x^2+y^2))-y(1+2(x^2+y^2))+a xy +b xy^2,\nonumber\\
  \dot y&=x(1+2(x^2+y^2))+y(1-(x^2+y^2))+c y^2+d x^3,
\end{align}
which has at the origin an unstable focus. Let us see that when
$b<8$ and $a,c$ and $d$ are such that
\[
\Psi_{a,b,c,d}(r):={-12}+\left( 2|2a-c|+10|c-a|\right)r+\left(2b-16
+12|b|+15|d|\right)r^2<0\]
for all $r>0,$ then it has at most one limit cycle, which when exists is hyperbolic and stable.

We apply Theorem \ref{MT}. Then $p(s)=((b-8)s+8)/8.$ By using
Proposition \ref{nueva} we take $w(r)=r^2p'(r^2)=(b-8)r^2/8$ and we
 choose $k=7/10.$ Then
\begin{align*}
M(r,\theta&,k,w)=\frac{3(b-8)}{40}r^2+\frac{8-b}{16}\left(
\frac{2a-c}5 \sin(\theta) +(c-a)\sin (3\theta)
\right)r^3\\&+\frac{b-8}{32}\left(\frac{16-2b}{5}+b\left(\frac{7}5\cos(2\theta)-\cos(4\theta)\right)+
d\left(2\sin(2\theta)+\sin(4\theta)\right)\right)r^4.
\end{align*}
Analogously that in the example given in the introduction, we have
\[
M(r,\theta,k,w)\le \frac{8-b}{160}r^2\Psi_{a,b,c,d}(r)<0
\]
for all $r>0$ and we are under the hypotheses of the theorem. Since
$m^+=1$ the uniqueness of the limit cycle follows. It is easy to see
that for $a,b,c$ and $d$ small enough the condition on
$\Psi_{a,b,c,d}$ holds and the limit cycle   exists.

Note that for this system we  can give a simple and  explicit
condition on the parameters of the system under which we can prove
that there is at most one limit cycle. For instance the condition
holds for $a=c=1$ and $-2b=2d=1.$ For these values we have also
checked numerically that the limit cycle actually exists.

\subsection{Example 2}
Consider the system
\begin{align}\label{eqex2}
  \dot x&=x(1-(x^2+y^2))(2-(x^2+y^2))-y+a x^2y +b x^2y^2,\nonumber\\
  \dot y&=x+y(1-(x^2+y^2))(2-(x^2+y^2))+cxy^2.
\end{align}

If $a,b$ and $c$ are such that
\[
\Psi_{a,b,c}(r):=-10+\frac94(|a|+|c|)+\frac94|b|r
+\left(12+|a|+|c|\right)r^2+|b|r^3-4r^4<0
\]
for all $r>0,$ then system \eqref{eqex2} has at most two
(hyperbolic)  limit cycles. Moreover, when they exist, one is
included in the disc $\mathcal{D}:=\{x^2+y^2\le3/2\}$ and is  stable
and the other one is outside the disc and is  unstable.

The proof follows again by using Theorem \ref{MT}. We take $p(s)=2-3s+s^2, $ and by
 Proposition \ref{nueva}, we consider $w(r)=r^2p'(r^2)=r^2(-3+2r^2)$ and
$k=1.$ Then
\begin{align*}
M(r,\theta,k,w)=&
\frac14\left(-40+a\left(6\sin(2\theta) -3\sin(4\theta)\right)+c\left(6\sin(2\theta) +3\sin(4\theta)\right)\right)r^4\\
&+\frac38b\left( 2\cos(\theta)-3\cos(3\theta)+\cos(5\theta)    \right)r^5\\
&+\left(12+a\sin(4\theta)-c\sin(4\theta)\right)r^6
-\frac{b}2\left(-\cos(3\theta)+\cos(5\theta)\right)r^7-4r^8.
\end{align*}
Hence, for the values of the parameters considered, we have
 \[
M(r,\theta,k,w)\le r^4\Psi_{a,b,c}(r)<0
\]
for all $r>0.$ Thus we can apply Theorem \ref{MT} with $m^+=2,$
proving the existence of at most two (hyperbolic) limit cycles.

For instance the condition on $\Psi_{a,b,c}$ holds for $a=1/8$,
$b=1/15$ and $c=1/20$. Moreover for these parameters it is not
difficult to prove, by using resultants and the Sturm's rule, that
the origin is the unique critical point, which  is unstable.
Finally, by studying the flow on $\{x^2+y^2=R^2\},$ for $R$ big
enough, and on $\{x^2+y^2=3/2\},$ we prove the existence of both
limit cycles.

\subsection{Example 3}
Consider the system
\begin{align}\label{eqex3}
  \dot x&=x(1-(x^2+y^2))(2-(x^2+y^2))-y(1-(x^2+y^2))+a x^4 +b x^2y^2,\nonumber\\
  \dot y&=x(1-(x^2+y^2))+y(1-(x^2+y^2))(2-(x^2+y^2)).
\end{align}
If $a$ and $b$ are such that
\[
\Psi_{a,b}(r):=-10+\frac{27|a|+9|b|}{4}r+12r^2+\left(2|a|
+|b|\right)r^3-4r^4<0\] for all $r>0,$  Theorem \ref{MT} will allow
us to show that system \eqref{eqex3} has at most two  limit cycles.
Moreover we will see that when there is no critical point outside
the disc $\mathcal{D}:=\{x^2+y^2\le3/2\}$ an unstable hyperbolic
limit cycle always exists outside the disc $\mathcal D$ and   it
surrounds several critical points.

We take $p, w$ and $k$ as in the previous example. Then
\begin{align*}
M(r,&\theta,k,w)=-10r^4+12r^6-4r^8\\
&+\frac38\left(a\left(14\cos(\theta) +3\cos(3\theta)-\cos(5\theta)\right)
+b\left(2\cos(\theta)-3\cos(3\theta)+\cos(5\theta)\right)\right)r^5\\
&+\frac12\left(a\left(-2\cos(\theta) +\cos(3\theta)+\cos(5\theta)\right)
+b\left(\cos(3\theta)-\cos(5\theta)\right)\right)r^7.
\end{align*}
Hence, when the conditions  on the parameters hold we have
\[
M(r,\theta,k,w)\le r^4\Psi_{a,b}(r)<0
\]
for all $r>0$ and we can apply Theorem \ref{MT}. In this case since
$m^+=2$ we know that system \eqref{eqex3} has at most two limit
cycles, and whenever they exist, one is inside the disc $\mathcal D$
and is hyperbolic and stable and the other one is outside the disc,
and is hyperbolic and unstable.  Since system \eqref{eqex3} has
 several critical points in the disc  $\mathcal D,$ $(0,0)$ and $(0,\pm1)$ among them,
 the unstable limit cycle, when it exists, surrounds these points.

An example of parameters for which $\Psi_{a,b}$ is negative is
$a=1/20$ and $b=1/15.$ Using the same tools that in the previous
case we can prove that the system has exactly five  critical points
and all them  are inside $\mathcal D.$ Hence by studying the flow on
the boundary of $\mathcal D$ and on $\{x^2+y^2=R^2\},$ for $R$ big
enough, we prove the existence of an unstable limit cycle
surrounding $\mathcal D.$

In short, for these values of the parameters we have proved that
system \eqref{eqex3} has  at most two limit cycles and the existence
of a limit cycle, which surrounds the five real critical points of
the system, which are in $\mathcal D.$ Our numerical explorations
indicate that this limit cycle is unique.

We want to stress that there are very few results in the literature
giving upper bounds for the number of limit cycles surrounding
several critical points. Notice that the critical points different
of the origin, surrounded by the limit cycle, come from the continua
of critical points $\{x^2+y^2=1\}$ that system \eqref{eqex3}
possesses when $a=b=0.$ The limit cycle is born in $\{x^2+y^2=2\}.$

\subsection{Example 4} Consider the system
\begin{align*}
  \dot x&=x(1-(x^2+y^2))(2-(x^2+y^2))(3-(x^2+y^2))-y(2-(x^2+y^2))+a x^2y^3,\\
  \dot y&=x(2-(x^2+y^2))+y(1-(x^2+y^2))(2-(x^2+y^2))(3-(x^2+y^2)).
\end{align*}
Let us prove that when $a$ is such that
\[
\Psi_{a}(r):=-98+\left(192+\frac{55}8|a|\right)r^2+\left(-144+6|a|\right)r^4+\left(48+\frac{3}2|a|\right)r^6-6r^{8}<0\]
for all $r>0,$ it has at most 3 limit cycles. Once more we use Theorem \ref{MT}. In this occasion we take  $k=1,$
$p(s)=(1-s)(2-s)(3-s)$ and $w(r)=r^2p'(r^2)=r^2(-11+12r^2-3r^4)=-3r^2(r^2-2-\frac{\sqrt3}3)(r^2-2+\frac{\sqrt3}3).$  Then
\begin{align*}
M(r,\theta,k,w)&=-98 r^4+\frac1{16}\left(3072 + 55 a\sin(2\theta) -44a\sin(4\theta)+11a\sin(6\theta)\right)r^6\\
&+\frac12 \left(-288 -3a\sin(2\theta) +6a\sin(4\theta)-3a\sin(6\theta)\right)r^8\\
&+\frac3{16} \left(256 -a\sin(2\theta) -4a\sin(4\theta)+3a\sin(6\theta)\right)r^{10}-6r^{12}.
\end{align*}
Since $m^+=3$ and $M(r,\theta,k,w)\le r^4\Psi_{a}(r)<0$ for all
$r>0,$   the existence of at most 3 limit cycles follows. For
instance the above hypothesis holds for $a=1/34.$ For this value of
the parameter we prove, by using the same tools that in the previous
examples, that the system has several critical points and that,
apart from the origin, all of them are contained in the ring
$\mathcal{C}=\{2-\frac{\sqrt3}3<x^2+y^2<2+\frac{\sqrt3}3\}.$
Finally, again similarly that in the previous cases, we prove that
there is exactly one  hyperbolic and stable limit cycle inside the
disc $\{x^2+y^2<2-\frac{\sqrt3}3\}$ and another hyperbolic and also
stable limit cycle outside the disc
$\{x^2+y^2\ge2+\frac{\sqrt3}3\}.$

The novelty of this example is the existence of a non-trivial system
for which the maximum number of limit cycles is known (it is 3).
Moreover it has  at least two hyperbolic limit cycles, one of them
surrounding only the origin and the other one surrounding the first
limit cycle and having several critical points between them.

We want to comment that our numerical exploration seems to indicate
that there is no limit cycle contained in~$\mathcal C,$ and so that
the maximum number of limit cycles of the system for this value of
the parameter is two.

\subsection{Example 5} This last example is  interesting from a theoretical point of view.
Consider the system
\begin{eqnarray}\label{triper}
  \dot x =&x u(x^2+y^2)-y v(x^2+y^2)+\varepsilon \widetilde{P}(x,y),\nonumber\\
  \dot y =&x v(x^2+y^2)+y u(x^2+y^2)+\varepsilon \widetilde{Q}(x,y),
\end{eqnarray}
where $u$ and $v$ are given real polynomials of degree $j$ and
assume $u$ is such that  all their roots are real and
simple. Let $m^+-1$ denote the number of positive real roots of $u'.$
Then, for any couple of polynomials $\widetilde{P}(x,y)$ and
$\widetilde{Q}(x,y)$ whose monomials have  degrees between $2$ and
$2j+1,$ both included, there exists
$\varepsilon_0=\varepsilon_0(\widetilde{P},\widetilde{Q} )>0$ such
that if $|\varepsilon|<\varepsilon_0$ then system \eqref{triper}
is under the hypotheses of Theorem \ref{MT}. Moreover, under these conditions, it
has at most $m^+$ limit cycles and all the existing limit cycles
are hyperbolic.

When $\varepsilon=0,$ the above assertions follow from
Lemma~\ref{tri} and Theorem~\ref{MT}. Consider $p(s)\equiv u(s).$ By
using Proposition~\ref{nueva} we can take $w(r)=r^2u'(r^2)$ and fix
a value $k>0$ for which
\[
p_{k,w}(r):=2 r^4\left(p(r^2)p''(r^2)-k(p'(r^2))^2\right)+ 2(1-k)r^2p(r^2)p'(r^2)<0
\]
for all $r\in(0,\infty).$ The function $p(s,\varepsilon)$ for
system~\eqref{triper} defined in Theorem~\ref{MT} writes as
$p(s,\varepsilon)=u(s)+\varepsilon \widetilde{p}(s),$ for some new
polynomial $\widetilde{p},$ also of degree $j.$ It is clear that for
$|\varepsilon|$ small enough, the polynomial $u(s)+\varepsilon
\widetilde{p}(s)$ has also all its roots real and simple and the
number of positive roots of  $u'(s)+\varepsilon \widetilde{p}\,'(s)$
is $m^+-1.$ By taking the same value of $k$ and
$w(r,\varepsilon)=r^2\left(u'(r^2)+\varepsilon
\widetilde{p}\,'(r^2)\right)$ we get that the function $M$ in the
theorem writes as
\begin{align*}
M(r,\theta,k,w(r,\varepsilon),\varepsilon):&=R(r,\theta,\varepsilon)\frac{\partial w(r,\varepsilon)}{\partial r}\\
&-k\left(\frac{\partial R(r,\theta,\varepsilon)}{\partial r}
+\frac{\partial \Theta(r,\theta,\varepsilon)}{\partial \theta}+\frac{R(r,\theta,\varepsilon)}{r}
\right)w(r,\varepsilon)\nonumber\\
&=:\sum_{i=2}^{4j} M_i(\theta,k,\varepsilon)r^i=p_{k,w}(r)+\varepsilon\sum_{i=2}^{4j}
\widetilde{M_i}(\theta,k,\varepsilon)r^i,
\end{align*}
where the functions $\widetilde{M_i}$ are smooth and $2\pi$-periodic
in $\theta.$ Since $p_{k,w}(r)$ is a negative polynomial in $(0,\infty)$
 of the form $p_{k,w}(r)=b_2r^2+b_4r^4+\cdots+b_{4j}r^{4j},$ for some
$b_i$ real numbers and,  moreover $b_2<0$ and $b_{4j}<0,$ it is
clear that for $|\varepsilon|$ small enough  the function
$\Phi_k(r,\varepsilon)$ given in Theorem~\ref{MT} is negative as we wanted to prove. Moreover
it follows that the maximum number of limit cycles of system~\eqref{triper} is
$m^+,$ and whenever they exist they are hyperbolic.

Notice that it is always true that when a planar system has $m^+$
hyperbolic limit cycles any small perturbation also has at least the
same number of hyperbolic limit cycles. The point of the above
example is not only to prove the existence of at most this number of
limit cycles, but to prove that there are planar polynomial systems
under the hypotheses of Theorem \ref{MT} for which all planar
systems of the same degree near them are also under the hypotheses
of the theorem. Moreover, as we can see in the study of the previous
examples, our approach allows  to get  explicit bounds for the size
of the perturbation under which our theorem can be applied.

\vspace{1cm}

 {\bf Acknowledgement.}  The first author is partially supported by a MEC/FEDER
grant number MTM2005-06098-C02-01 and by a CICYT grant number
2005SGR 00550. He also  wants to thank  the Laboratoire de
Ma\-th\'{e}\-ma\-tique et Physique Th\'{e}orique of the  Universit\'{e} de Tours
for their support and hospitality during the period in which this
paper was written.


\begin{thebibliography}{99}




\bibitem{Cherkas0} {\sc L. A. Cherkas}, \emph{Estimation of the number
of limit cycles of autonomous systems}, Differential Equations \textbf{13},
529--547, (1977).

\bibitem{Cherkas} {\sc L. A. Cherkas}, \emph{Dulac function for
polynomial autonomous systems on a plane}, Differential Equations \textbf{33},
692--701, (1997).

\bibitem{CG1} {\sc L. A. Cherkas, A. A. Grin'}, \emph{A second-degree polynomial
Dulac function for a cubic system on the plane}, Differential Equations
\textbf{33}, 1443--1445, (1997).

\bibitem{CG2} {\sc L. A. Cherkas, A. A. Grin'}, \emph{ A Dulac function in a half-plane
in the form of a polynomial of
the second degree for a quadratic system},  Differential Equations \textbf{34},
1346--1348, (1998).

\bibitem {DLA} {\sc F. Dumortier, J. Llibre and J.C. Art\'{e}s},
``Qualitative theory of planar differential systems'', UniversiText,
Springer--Verlag, New York, 2006.

\bibitem{GGi} {\sc A. Gasull, H. Giacomini},
\emph{ A new criterion for controlling the number of limit cycles of
some generalized Li\'{e}nard equations}, J. Differential Equations,
\textbf{185},  54–-73 (2002).

\bibitem{GG} {\sc A. Gasull, A. Guillamon}, \emph{Non-existence, uniqueness of
limit cycles and center problem in a system that includes
predator-prey systems and generalized Li\'enard equations}, Diff.
Equations and Dynamical Systems, \textbf{3}, 345--366 (1995).



\bibitem{Lloyd} {\sc N. G. Lloyd}, \emph{A note on the number of
limit cycles in certain  two-dimensional systems}, J. London Math. Soc. (2),
\textbf{20},   277--286, (1979).




\bibitem{Yamato} {\sc K. Yamato}, \emph{An effective method of counting the
number of limit cycles}, Nagoya Math. J. \textbf{76}, 35--114 , (1979).

\bibitem{Ye} {\sc Yan Qian Ye \& others}, ``Theory of limit cycles'', Translations of
Mathematical Monographs \textbf{66}. American Mathematical Society, Providence,
RI, 1986.

\bibitem{Z}
{\sc  Zhi Fen Zhang \& others}, ``Qualitative theory of differential
equations'', Translations of Mathematical Monographs \textbf{101}.
American Mathematical Society, Providence, RI, 1992.

\vspace{1cm}

\end{thebibliography}
\end{document}